\documentclass[12pt,reqno]{article}

\usepackage[usenames]{color}
\usepackage[colorlinks=true,
linkcolor=webgreen, filecolor=webbrown,
citecolor=webgreen]{hyperref}

\definecolor{webgreen}{rgb}{0,.5,0}
\definecolor{webbrown}{rgb}{.6,0,0}

\usepackage{amssymb}
\usepackage{graphicx}
\usepackage{amscd}
\usepackage{lscape}
\usepackage{tikz}
\usepackage{tikz-cd}
\usepackage{pgfplots}

\usetikzlibrary{matrix}
\usetikzlibrary{fit,shapes}
\usetikzlibrary{positioning, calc}
\tikzset{circle node/.style = {circle,inner sep=1pt,draw, fill=white},
        X node/.style = {fill=white, inner sep=1pt},
        dot node/.style = {circle, draw, inner sep=5pt}
        }
\usepackage{tkz-fct}

\usepackage{amsthm}
\newtheorem{theorem}{Theorem}

\newtheorem{proposition}[theorem]{Proposition}
\newtheorem{corollary}[theorem]{Corollary}

\theoremstyle{definition}

\newtheorem{example}[theorem]{Example}

\usepackage{float}

\usepackage{graphics,amsmath}
\usepackage{amsfonts}
\usepackage{latexsym}
\usepackage{epsf}

\newcommand{\seqnum}[1]{\href{http://oeis.org/#1}{\underline{#1}}}

\begin{document}

\begin{center}
\vskip 1cm{\LARGE\bf A Three-parameter Family Of Involutions In The Riordan Group Defined By Orthogonal Polynomials} \vskip 1cm \large
Paul Barry\\
School of Science\\
Waterford Institute of Technology\\
Ireland\\
\href{mailto:pbarry@wit.ie}{\tt pbarry@wit.ie}
\end{center}
\vskip .2 in

\begin{abstract} We show how to define, for every Riordan group element $(g(x), f(x))$, an involution in the Riordan group. More generally, we show that for every pseudo-involution $P$ in the Riordan group, we can define a new involution beginning with an arbitrary element $(g(x), f(x))$ in the Riordan group. We then use this result to show that certain two-parameter families of orthogonal polynomials  defined by a Riordan array can lead to involutions in the Riordan group, and we give an explicit form of these involutions.
\end{abstract}

\section{Preliminaries}
The Riordan group $\mathcal{R}$ \cite{book, SGWW, bbook} is a subgroup of the group of lower-triangular invertible matrices, whose elements are defined algebraically. We work over a field of characteristic zero $\mathbf{k}$, which can be taken to be the field $\mathbb{C}$ of complex numbers. In practice, many of the examples we shall give will have entries taken in the ring of integers $\mathbb{Z}$. The group $\mathcal{R}$ can be viewed as an abstract group, or as its matrix realization. We explain this now.
We let 
$$\mathbf{F}_0 = \{ g(x)=g_0+g_1 x + g_2 x^2+ \cdots\,|\, g_0 \ne 0, g_i \in \mathbf{k} \},$$ be the ring of \emph{invertible} power series in the indeterminate $x$ over $\mathbf{k}$. Similarly we let 
$$\mathbf{F}_1 = \{ f(x)=f_1 x + f_2 x^2+ \cdots\,|\, f_0 = 0, f_1 \ne 0, f_i \in \mathbf{k} \}$$ be the ring of \emph{composable} power series in the indeterminate $x$  over the field $\mathbf{k}$. 
Then as an abstract group, we have 
$$\mathcal{R}=\{ (g(x), f(x))\,|\, g \in \mathbf{F}_0, f \in \mathbf{F}_1\}.$$ 
In fact, we have that 
$$\mathcal{R} = \mathbf{F}_0 \rtimes \mathbf{F}_1.$$ 
The matrix representation of the group $\mathcal{R}$ can be obtained by associating to the pair of power series $(g(x), f(x))$ the matrix $(a_{n,k})_{0 \le n,k \le \infty}$, with $a_{n,k} \in \mathbf{k}$, where 
$$a_{n,k}=[x^n] g(x)f(x)^k,$$ with $[x^n]$ denoting the functional \cite{MC} that extracts the coefficient of $x^n$ is the power series to which it is applied. The fact that $f(x) \in \mathcal{F}_1$ ensures that this matrix is lower triangular. We use the term \emph{Riordan array} to denote the matrix element corresponding to the abstract element $(g(x), f(x))$.

The group structure on $\mathcal{R}$ is defined by the product 
$$(g(x), f(x))\cdot (u(x), v(x))= (g(x).u(f(x)), v(f(x))),$$ and the inverse  
$$(g(x), f(x))^{-1}= \left(\frac{1}{g(\bar{f}(x))}, \bar{f}(x)\right),$$ where 
$\bar{f}(x) = f^{\langle -1 \rangle}$ is the compositional inverse of $f$. This means that $\bar{f}$ is the solution $u(x)$ of the equation $f(u)=x$ which satisfies $u(0)=0$. 

The identity element is $(1, x)$. The product in $\mathcal{R}$ corresponds to matrix multiplication in its matrix representation. 

\begin{example} Using the product rule, we have 
$$(g(x), f(x))= (g(x), x) \cdot (1, f(x)),$$ which corresponds to $\mathcal{R} = \mathbf{F}_0 \rtimes \mathbf{F}_1$.
\end{example}

\begin{example} The Riordan group element $\left(\frac{1}{1-x}, \frac{x}{1-x}\right)$ is represented by the binomial matrix (Pascal's triangle) $\left(\binom{n}{k}\right)_{0 \le n,k \le \infty}$.
\end{example} 

In the rest of this note, we shall drop the distinction between abstract elements such as $\left(\frac{1}{1-x}, \frac{x}{1-x}\right)$ and their representations (here, $\left(\binom{n}{k}\right)_{0 \le n,k \le \infty}$), and we shall use the most convenient form for the topic under discussion.

\section{Involutions and pseudo-involutions in the Riordan group}
By an \emph{involution} in the Riordan group \cite{B2, Burstein, CN, CJKS, CK, CKS07, He1, HSh2, PS} we understand an element $(g(x), f(x)) \in \mathcal{R}$ such that
$$ (g(x), f(x))^2 = (g(x), f(x)) \cdot (g(x), f(x))=(g(x).g(f(x)),f(f(x)) = (1,x).$$ We see immediately that for an involution $(g(x), f(x))$,  we have 
$$ \bar{f}=f$$ and 
$$ g(x)= \frac{1}{g(f(x))} = \frac{1}{g(\bar{f}(x))}.$$ 
In particular, we have $(g(x), f(x))^{-1}  = (g(x), f(x))$. 

\begin{example} The Riordan arrays $(1, x)$ and $(1,-x)$ are involutions.
\end{example}
\begin{example} The signed binomial matrix $\left(\frac{1}{1-x}, \frac{-x}{1-x}\right)$ or $\left((-1)^k \binom{n}{k}\right)$ is an involution in the Riordan group.
\end{example}
By a \emph{pseudo-involution} in the Riordan group we shall understand an element $(g(x), f(x))$ such that 
$$(g(x), f(x))\cdot (1, -x) = (g(x), -f(x))$$ is an involution. 

We see for instance that the binomial matrix $\left(\frac{1}{1-x}, \frac{x}{1-x}\right)$ is thus a pseudo-involution. 
\begin{example} The identity $(1,x)$ is trivially a pseudo-involution. This is because 
$(1, x) \cdot (1,-x)=(1,-x)$ is an involution.
\end{example}

\begin{example} The set of elements 
$$\mathbf{Bin}=\left\{ \left(\frac{1}{1-\alpha x}, \frac{x}{1-\alpha x}\right)\,|\, \alpha \in \mathbf{k} \right\}$$ is a subgroup of the Riordan group. All its elements are pseudo-involutions.
\end{example}

\section{Orthogonal polynomials and Riordan arrays}
By a family of orthogonal polynomials $P_n(x)$ we shall understand a sequence of polynomials $P_n(x)$ of exact degree $n$, $n \ge -1$, such that they satisfy a three-term recurrence 
$$P_n(x)=(x- \alpha_n)P_{n-1}(x) - \beta_n P_{n-2}(x),$$ with $P_{-1}(x)=0$ and $P_0(x)=1$.    
If we let $P_n(x)=\sum_{i=0}^n a_{n,i}x^i$, then it is clear that the matrix $(a_{n,k})$ is lower triangular. The question then arises as to whether a Riordan array can be the coefficient matrix of a family of orthogonal polynomials. The answer is in the affirmative \cite{classical, bbook}. We have that the Riordan array 
$$\left(\frac{1- rx- s x^2}{1+ ax + b x^2}, \frac{x}{1+ ax+bx^2}\right)$$ is the coefficient array of the family of (constant coefficient) generalized Chebyshev polynomials $P_n(x)$ that satisfy the three-term recurrence 
$$P_n(x) = (x-a) P_{n-1}(x)- b P_{n-2}(x),$$ with 
$P_0(x)=1$,  $P_1(x)=x-a-r$ and $P_2(x)=x^2-(2a+r)x+a^2+ar-b-s$.   

\section{The main result}
Our results concerning orthogonal polynomials will be a consequence of the following general results.

\begin{proposition} We let $P$ denote an arbitrary pseudo-involution in the Riordan group. For arbitrary $(g(x), f(x)) \in \mathcal{F}_0 \times \mathcal{F}_1$  the Riordan array
$$ (g(x), f(x))^{-1} \cdot P \cdot (g(-x), f(-x))$$ is an involution.
\end{proposition}
\begin{proof}
We have  $$(g(-x), f(-x))=(1,-x) \cdot (g(x), f(x)).$$
Thus 
$$(g(x), f(x))^{-1} \cdot P \cdot (g(-x), f(-x))=(g(x), xg(x))^{-1} \cdot P \cdot (1,-x)\cdot (g, xg(x)).$$
Now $ \mathcal{I}= P \cdot (1,-x)$ is an involution, and so 
$$(g(x), f(x))^{-1} \cdot P \cdot (g(-x), f(-x))= (g(x), f(x))^{-1} \cdot \mathcal{I} \cdot (g(x), f(x))$$ is the conjugate of an involution, which is again an involution.
\end{proof}

\begin{example}
We take the example of $g(x)=c(x)=\frac{1-\sqrt{1-4x}}{2x}$, the generating function of the Catalan numbers $C_n=\frac{1}{n+1}\binom{2n}{n}$ \seqnum{A000108}. We let $f(x)=xg(x)$, and  we take  $P=(1,x)$, the identity, which is both an involution and a pseudo-involution.
Then we obtain the involution 
$$ ((1+ x c(x))c(x), -x(1+xc(x))c(x)).$$
\end{example}

\begin{example} We let $(g(x), f(x))= \left(\frac{1}{1+rx+sx^2}, \frac{x}{1+rx+sx^2}\right)$. Then 
$$(g(x), f(x))^{-1} \cdot (g(-x),-f(-x)) = \left(\frac{1}{1-2rx}, \frac{-x}{1-2rx}\right),$$ an element of the subgroup $\mathbf{Bin}$. 
\end{example}

\begin{example} The array product  
$$\left(\frac{1}{1+x^2}, \frac{x(1-x)}{1+x^2}\right)^{-1}\cdot \left(\frac{1}{1+x^2}\frac{-x(1+x)}{1+x^2}\right)$$ yields the involution 
$$(1, -x(1+x)c(x(1+x))).$$
\end{example}

\begin{example} \textbf{The RNA matrix}. Cameron and Nkwanta \cite{CN} give the example of the RNA matrix. Here, we take $(g, f)=\left(\frac{1}{1+x^2}, \frac{x}{1+x^2}\right)^{-1}$, and $P=\left(\frac{1}{1-x},\frac{x}{1-x}\right)$. Thus the RNA involution matrix is given by the product
$$\left(\frac{1}{1+x^2}, \frac{x}{1+x^2}\right) \cdot \left(\frac{1}{1-x},\frac{x}{1-x}\right) \cdot \left(\frac{1}{1+x^2}, \frac{-x}{1+x^2}\right)^{-1}.$$ The matrix begins 
$$\left(
\begin{array}{ccccccc}
 1 & 0 & 0 & 0 & 0 & 0 & 0 \\
 1 & -1 & 0 & 0 & 0 & 0 & 0 \\
 1 & -2 & 1 & 0 & 0 & 0 & 0 \\
 2 & -3 & 3 & -1 & 0 & 0 & 0 \\
 4 & -6 & 6 & -4 & 1 & 0 & 0 \\
 8 & -13 & 13 & -10 & 5 & -1 & 0 \\
 17 & -28 & 30 & -24 & 15 & -6 & 1 \\
\end{array}
\right).$$ The unsigned version, which is a pseudo-involution, is \seqnum{A097724}.
\end{example}

\begin{example} Our final example of this section uses the matrix $(g(x), f(x))=(M(x), xM(x))$ where 
$M(x)= \frac{1-x-\sqrt{1-2x-3x^2}}{2x^2}$ is the generating function of the Motzkin numbers \seqnum{A001006}. 
Then the product $(M(x), xM(x))^{-1} \cdot (M(-x),-xM(-x))$ yields the involution 
$$\left(\frac{1}{(1+x)^2} c\left(\frac{x^2}{(1+x)^4}\right), \frac{-x}{(1+x)^2} c\left(\frac{x^2}{(1+x)^4}\right)\right).$$
This involution begins 
$$\left(
\begin{array}{ccccccc}
 1 & 0 & 0 & 0 & 0 & 0 & 0 \\
 -2 & -1 & 0 & 0 & 0 & 0 & 0 \\
 4 & 4 & 1 & 0 & 0 & 0 & 0 \\
 -10 & -12 & -6 & -1 & 0 & 0 & 0 \\
 28 & 36 & 24 & 8 & 1 & 0 & 0 \\
 -82 & -112 & -86 & -40 & -10 & -1 & 0 \\
 248 & 356 & 300 & 168 & 60 & 12 & 1 \\
\end{array}
\right).$$ 
\end{example}

\section{Orthogonal polynomials and involutions}

We consider the Riordan array 
$$(g,f)=\left(\frac{1}{1+rx+sx^2}, \frac{x(1-tx)}{1+rx+sx^2}\right).$$ 
Taking $P=(1,x)$, this gives us the  involution $(\tilde{g}(x), \tilde{f}(x))$ in the Riordan group 
where 
$$\tilde{g}(x)=\frac{s+t(r+t)}{s+t(t+1)-(r^2t+2rs-s+t^2)x+t(r-1)\sqrt{1-2(r+2t)x+(r^2-4s)x^2}},$$ and 
$$\tilde{f}(x)= \frac{t\sqrt{1-2(r+2t)x+(r^2-4s)x^2}+(t^2-s)x-t}{s+t^2-r(rt+2s)x+rt\sqrt{1-2(r+2t)x+(r^2-4s)x^2}}.$$

The Riordan array $\left(\frac{1+(r-s)x}{1+(r+s)x}, \frac{x}{(1+rx)(1+(r+s)x)}\right)$ is the coefficient array of the family of orthogonal polynomials $P_n(x;r,s)$ that satisfy the $3$-term recurrence
$$P_n(x;r,s)=(x-(2r+s)) P_{n-1}(x;r,s)-r(r+s) P_{n-2}(x;r,s),$$ with 
$P_0(x;r,s)=1$ and $P_1(x;r,s)=x-2s$.

We then have the following proposition.

\begin{proposition}
For $r,s \in \mathbb{Z}$, the Riordan array 
\begin{scriptsize}
$$\left(\frac{1-(r-2t)x-(rt-s-t^2)x^2}{1+(r+2t)x+(rt+s+t^2)x^2}, \frac{x}{1+(r+2t)x+(rt+s+t^2)x^2)}\right)^{-1} \cdot \left(1,\frac{-x(1+2tx)}{1-(r-2t)x-(rt-s-t^2)x^2}\right)$$ 
\end{scriptsize} is an involution in the Riordan group.
\end{proposition}

\begin{proof}
Evaluating the product, we find that it is equal to $(\tilde{g}(x), \tilde{f}(x))$.
\end{proof}

Thus for each triple $(r,s,t) \in \mathbf{Z}^3$, we can associate to the family of orthogonal polynomials with coefficient array the Riordan array   $\left(\frac{1+(r-2t)x-(rt-s-t^2)x^2}{1+(r+2t)x+(rt+s+t^2)x^2}, \frac{x}{1+(r+2t)x+(rt+s+t^2)x^2)}\right)$ an involution in the Riordan group.

We note that the generating function $\tilde{g}(x)$ has the following generating function expression.
$$\tilde{g}(x)=
\cfrac{1}{1-2rx-\cfrac{2rtx^2}{1-(r+2t)x-\cfrac{(s+t(r+t))x^2}{1-(r+2t)x-\cdots}}},$$ while the row sums of the involution $(\tilde{g}(x), \tilde{f}(x))$ have a generating function that can be expressed as the related continued fraction 
$$
\cfrac{1}{1-(2r-1)x-\cfrac{2t(r-1)x^2}{1-(r+2t)x-\cfrac{(s+t(r+t))x^2}{1-(r+2t)x-\cdots}}}.$$

\begin{example}
We let $s=t(r-t), r=1, t=\frac{1}{2}$. Then we obtain that 
$$\left(\frac{1}{(1+x)^2}, \frac{x}{(1+x)^2}\right)^{-1} \cdot (1, -x(1+x)) = (c(x)^2, -xc(x)^3)$$ is an involution.

A variant of this is found by setting $r=2t$ and $s=t^2$ to obtain that 
$$\left(\frac{1}{(1+2tx)^2}, \frac{x}{(1+2tx)^2}\right)^{-1} \cdot (1,-x(1+2tx))$$ is an involution. This shows that 
$$ (c(2tx)^2, -x ct(2tx)^3)$$ is an involution.
\end{example}

\begin{corollary}
Given $r,t \in \mathbb{Z}$, the Riordan array 
$$\left(\frac{1+(t-r)x}{1+(t+r)x}, \frac{x}{(1+tx)(1+(t+r)x)}\right)^{-1} \cdot \left(1, \frac{-x(1+2tx)}{(1+tx)(1+(t-r)x)}\right)$$ is an involution in the Riordan group.
\end{corollary}
\begin{proof} This is the case $s=0$ of the above proposition.
\end{proof}

\begin{example}
We take the case of $r=t=1$ to obtain the involution 
$$\left(\frac{1}{1+2x}, \frac{x}{(1+x)(1+2x)}\right)^{-1}\cdot \left(1, \frac{-x(1+2x)}{1+x}\right)=(S(x), -xS(x)^2),$$ where $S(x)$ denotes the generating function of the large Schr\"oder numbers \seqnum{A006318}.
This array begins 
$$\left(
\begin{array}{ccccccc}
 1 & 0 & 0 & 0 & 0 & 0 & 0 \\
 2 & -1 & 0 & 0 & 0 & 0 & 0 \\
 6 & -6 & 1 & 0 & 0 & 0 & 0 \\
 22 & -30 & 10 & -1 & 0 & 0 & 0 \\
 90 & -146 & 70 & -14 & 1 & 0 & 0 \\
 394 & -714 & 430 & -126 & 18 & -1 & 0 \\
 1806 & -3534 & 2490 & -938 & 198 & -22 & 1 \\
\end{array}
\right).$$ 
The unsigned matrix is \seqnum{A110098}.

We note that the $B$-sequence of this array is $4,4,4,\ldots$. 

For the general case of the involution 
$$\left(\frac{1}{1+2rx}, \frac{x}{(1+2rx)(1+rx)}\right)^{-1} \cdot \left(1, -\frac{x(1+2rx)}{1+rx}\right)$$ we observe that the $B$-sequence begins $4r, 4r^3, 4r^5,\ldots$.
\end{example}

\begin{example} We take the case of $s=0, r=1, t=2r=2$. We find that the product 
$$\left(\frac{1+x}{1+3x}, \frac{x}{(1+2x)(1+3x)}\right)^{-1} \cdot \left(1, \frac{-x(1+4x)}{(1+x)(1+2x)}\right)$$ gives the involution 
$$\left(\frac{3}{2-x+\sqrt{1-10x+x^2}}, -\frac{1-5x+x^2-(1-x)\sqrt{1-10x+x^2}}{1+2x}\right).$$ 
The corresponding pseudo-involution begins 
$$\left(
\begin{array}{ccccccc}
 1 & 0 & 0 & 0 & 0 & 0 & 0 \\
 2 & 1 & 0 & 0 & 0 & 0 & 0 \\
 8 & 8 & 1 & 0 & 0 & 0 & 0 \\
 44 & 56 & 14 & 1 & 0 & 0 & 0 \\
 284 & 404 & 140 & 20 & 1 & 0 & 0 \\
 2012 & 3044 & 1268 & 260 & 26 & 1 & 0 \\
 15140 & 23804 & 11132 & 2852 & 416 & 32 & 1 \\
\end{array}
\right).$$  
\end{example}

\section{Conclusions} The central result of this note is that for any Riordan array $(g,f)$, and any pseudo-involution $P$, we have an involution $\mathcal{I}(g,f,P)$ in the Riordan group, where 
$$\mathcal{I}(g,f,P) = (g(x), f(x))^{-1} \cdot P \cdot (g(-x), f(-x)).$$ 
We have investigated this in the specific case of the coefficient matrices of certain families of orthogonal 
polynomials that can be defined by Riordan arrays. Clearly, other families of Riordan arrays $(g(x), f(x))$ could be investigated. In our examples, we have confined our discussions to the cases $P=(1,x)$ and $P=\left(\frac{1}{1-x}, \frac{x}{1-x}\right)$. Again, it is clear that other choices of $P$ may lead to other interesting examples. With $\mathcal{I}=\mathcal{I}(g,f,P)$, we have expressions such as 
$$(g(x), f(x)) \cdot \mathcal{I} = P \cdot (g(-x), f(-x))$$ and 
$$ (g(x), f(x)) = P \cdot (g(-x), f(-x)) \cdot \mathcal{I},$$ which may warrant further consideration.

\section{Acknowledgments} The author has benefitted from fruitful discussions with Lou Shapiro.

\bigskip
\hrule
\bigskip
\noindent 2020 \emph{Mathematics Subject Classification}: Primary 05A15. Secondary 05A05, 05A19, 11B83.
\noindent \emph{Keywords:} Riordan array, Riordan group, involution, pseudo involution, orthogonal polynomial.

\bigskip
\hrule
\bigskip
\noindent (Concerned with sequences
\seqnum{A000108},
\seqnum{A001006},
\seqnum{A006318},
\seqnum{A097724}, and
\seqnum{A110098}.)

\end{document}